\documentclass{amsart}

\usepackage{amsmath}
\usepackage{booktabs}
\usepackage{latexsym}
\usepackage{mathrsfs}
\usepackage{amsthm}
\usepackage{amssymb}
\usepackage{amsfonts}
\usepackage{amsbsy}

\usepackage[shortlabels]{enumitem}
\usepackage{url}
\usepackage{array}
\usepackage{pdflscape}
\usepackage{xcolor}
\usepackage{stmaryrd}
\usepackage{verbatim}
\usepackage{braket}
\usepackage{graphicx}
\usepackage{tikz}

\newcommand{\Cc}{{\mathbb C}}

\newcommand{\Rr}{{\mathbb R}}

\def\tr{\operatorname{tr}}

\theoremstyle{plain}
\newtheorem{thm}{Theorem}
\newtheorem{lem}[thm]{Lemma}
\newtheorem{prop}[thm]{Proposition}

\newtheorem{defn}{Definition}
\newtheorem{example}{Example}

\def\tr{\operatorname{tr}}
\usepackage[colorlinks=true]{hyperref}

\numberwithin{equation}{section}

\begin{document}
	
	\title{On the exact quantum chromatic number of generalized Johnson graphs}
	
	\author{Gaojun Luo}
	\address{The School of Mathematics, Nanjing University of Aeronautics and Astronautics, Nanjing 210016, China}
	\email{gaojun\_luo@nuaa.edu.cn}
	\thanks{The first author was supported by National Natural Science Foundation of China under Grant 12401690.}
	
	
	\author{Xiwang Cao}
	\address{The School of Mathematics, Nanjing University of Aeronautics and Astronautics, Nanjing 210016, China}
	\email{xwcao@nuaa.edu.cn}
	\thanks{The second author was supported by National Natural Science Foundation of China under Grant 12571575.}
	
	\author{Shitao Li}
	\address{the School of Internet, Anhui University, Hefei, Anhui
		230039, China}
	\email{lishitao0216@163.com}
	
	\author{Yang Li}
	\address{The School of Physical and Mathematical Sciences, Nanyang
		Technological University, 637371, Singapore}
	\email{yanglimath@163.com}
	
	
	
	
	
	\begin{abstract}
The quantum chromatic number is a fundamental parameter in the study of nonlocal games, capturing the extent to which entanglement can improve performance in distributed tasks. In this paper, we investigate the quantum chromatic number of generalized Johnson graphs. By constructing modulus-one orthogonal representations, we obtain general upper bounds on their quantum chromatic numbers. We further analyze the smallest eigenvalue of these graphs. Combining the resulting Hoffman-type lower bounds with the upper bounds obtained from orthogonal representations, we determine the exact quantum chromatic numbers of two infinite families of generalized Johnson graphs. Finally, applying a forbidden-distance theorem for binary codes, we show that the classical chromatic numbers of these families grow exponentially with $n$, whereas their quantum chromatic numbers grow linearly. These families exhibit an exponential separation between the classical and quantum chromatic numbers.
	\end{abstract}
	
	\maketitle
	
	\section{Introduction}

Quantum entanglement allows spatially separated players to
produce correlations that cannot be reproduced by classical strategies without
communication. Brassard, Broadbent and Tapp in~\cite{Brassard2005}
surveyed this phenomenon under the name quantum pseudo-telepathy, while
Cleve et al. in~\cite{Cleve2004} investigated the
consequences and limitations of nonlocal quantum strategies. Given a
graph $G=(V,E)$ and a positive integer $c$, a verifier sends vertices $u$ and
$v$ to two noncommunicating players, Alice and Bob, respectively. The players
must return the same color when $u=v$ and different colors when $uv\in E$. Classically, a perfect strategy exists if and only if
$c$ is at least the classical chromatic number $\chi(G)$. When the players share an entangled state and perform local
measurements, they may succeed with fewer colors. The least such number of
colors is called the quantum chromatic number of $G$ and is denoted by
$\chi_Q(G)$.

Cameron et al. in \cite{Cameron2007} gave a systematic mathematical treatment of the quantum chromatic number. They also established connections between the quantum chromatic number, clique numbers and orthogonal representations. Man\v{c}inska and Roberson in \cite{Mancinska2016} developed the framework of quantum graph homomorphisms. In this framework, a quantum coloring is viewed as a quantum homomorphism to a complete graph. In~\cite{Mancinska2016a}, the same authors constructed a graph on $14$ vertices whose quantum chromatic number is strictly smaller than its classical chromatic number. Lalonde in~\cite{Lalonde2025} proved that $14$ is the smallest possible order of a graph with such a separation. Hadamard graphs form a main family that exhibits a quantum advantage. For the Hadamard graph $H_n$, Brassard, Cleve and Tapp in~\cite{Brassard1999} proved that $\chi_Q(H_n)\leq n$ when $n$ is a power of $2$. Avis, Hasegawa, Kikuchi and Sasaki in \cite{Avis2006} extended this result to every positive integer $n$ divisible by four. For classical coloring, a result of Frankl and R\"odl in~\cite{Frankl1987} implies that the chromatic number of Hadamard graphs grows exponentially with $n$. Thus, Hadamard graphs exhibit an exponential separation between the classical and quantum chromatic numbers.

In general, determining the exact value of $\chi_Q(G)$ is difficult. Ji in \cite{Ji2013} proved that computing the quantum chromatic number of a general graph is NP-hard. At present, the exact value of $\chi_Q(G)$ is known for only a few infinite families of graphs. Cameron et al. in \cite{Cameron2007} showed that the complete graph on $n$ vertices has quantum chromatic number $n$ and that every nonempty bipartite graph has quantum chromatic number $2$. It was proved in \cite{Mancinska2016,Elphick2019,McNamara2024} that $\chi_Q(H_n)=n$ when $n$ is a multiple of $4$. Cao, Feng and Tan in~\cite{Cao2024} determined the quantum chromatic number of the binary Hamming graphs $H(4t-1,2,2t)$. Cao et al. in \cite{Cao2025} determined the quantum chromatic number for a class of generalized Hadamard graphs and quantified the separation between their quantum and classical chromatic numbers. Luo, Ning and Zhang in \cite{Luo2025} determined the quantum chromatic number of certain subgraphs of orthogonality graphs.

In this paper, we investigate the quantum chromatic number of the
generalized Johnson graph $J(n,k,k-r)$. In Section~3.1, we construct
modulus-one orthogonal representations and apply the bound of Cameron
et al.~\cite{Cameron2007} to derive upper bounds on
$\chi_Q(J(n,k,k-r))$. In Section~3.2, we study the
eigenvalues of generalized Johnson graphs. It is known in \cite{Brouwer2018} that $E_r(1)$ is the smallest
eigenvalue if and only if
$
r(n-1)\ge k(n-k).
$
We determine additional conditions under which $E_r(2)$ is the
smallest eigenvalue. Combining these spectral results with the
Hoffman-type bound of Elphick and Wocjan~\cite{Elphick2019}, we obtain
lower bounds on the quantum chromatic number. Matching the upper and
lower bounds yields exact values for two infinite families of
generalized Johnson graphs.
\begin{thm}\label{thm-exact-1} 
	Let $n,k,r$ be positive integers with $1\le r\le\min\{k,n-k\}$, $r(n-1)=k(n-k)$ and $r\ge\frac n4$. Then we have
	\[
	\chi_Q(J(n,k,k-r))=n.
	\]
\end{thm}

\begin{thm}\label{thm-exact-2}
	Let $a$ and $r$ be integers with $r\ge 5$ and $2\le a \le \sqrt r$. Then we have
	\[
	\chi_Q(J(4r,2r-a,r-a))=4r
	\]
	and
	\[
	\chi_Q(J(4r,2r+a,r+a))=4r.
	\]
\end{thm}

In contrast to the linear growth of the quantum chromatic number, we show that the classical chromatic number of $J(n,k,k-r)$ grows exponentially with $n$. We also quantify the separation between the quantum and classical chromatic numbers of these graphs.

\begin{thm}\label{thm-s}
	Let $n,k,r$ be positive integers with $1\le r\le\min\{k,n-k\}$. Let $J(n,k,k-r)$ be the generalized Johnson graph defined in Theorem \ref{thm-exact-1} or Theorem \ref{thm-exact-2}. Then we have $\chi_Q(J(n,k,k-r))<\chi(J(n,k,k-r))$ if $n$ is large enough.
\end{thm}

	\section{Preliminaries}\label{sec:pre}
	
	Given a positive integer $n$, let $[n]=\{1,2,\ldots,n\}$. For $0\le k\le n$, let $\binom{[n]}{k}$ denote the set of all $k$-element subsets of $[n]$. Given two sets $A,B\subseteq[n]$, let $A\triangle B=(A\setminus B)\cup(B\setminus A)$ be the symmetric difference of $A$ and $B$. Let  $\binom ab=0$ if $b>a$ or $b<0$.
	
	\subsection{Generalized Johnson graph}
	
	Let $n,k,s$ be integers with $0\le s\le k\le n$.  The generalized Johnson graph $J(n,k,s)$
	is the graph whose vertex set is
	\[
	V(J(n,k,s))=\binom{[n]}{k}
	\]
	and where two distinct vertices $A,B\in\binom{[n]}k$ are adjacent if and only if $|A\cap B|=s$.
	
	Since  $|A\triangle B|=2(k-|A\cap B|)$,  two distinct vertices $A,B$ are adjacent if and only if $|A\triangle B|=2(k-s)$. The graph $J(n,k,s)$ is a distance-$r$ graph in the Johnson scheme with $r=k-s$. Throughout this paper, we often write it as $J(n,k,k-r)$. 
	
	Define the map from $\binom{[n]}k$ to $\binom{[n]}{n-k}$ by 
	\[
	A\longmapsto A^c=[n]\setminus A.
	\]
	If $|A\cap B|=s$, then we have
	\[
	|A^c\cap B^c|
	=n-|A\cup B|
	=n-(2k-s)
	=n-2k+s.
	\]
	Therefore, the graph $J(n,k,k-r)$ is isomorphic to $J(n,n-k,n-k-r)$. Since quantum chromatic number is invariant under graph isomorphism, we assume without loss of generality that $k\le \frac n2$.
	
	Next we recall the eigenvalues of $J(n,k,k-r)$. Given a $k$-set $A$, a neighbor is obtained by removing $r$ elements from $A$ and adding $r$ elements from $[n]\setminus A$. The graph $J(n,k,k-r)$ is regular
	with degree
	\[
	\Delta_r=\binom{k}{r}\binom{n-k}{r}.
	\]
	According to \cite{Godsil2016}, for each $j=0,1,\ldots,k$, the $j$-th eigenvalue of $J(n,k,k-r)$ is the Eberlein polynomial
	\begin{equation}\label{eq:eberlein}
		E_r(j)
		=
		\sum_{\ell=0}^{r}
		(-1)^\ell
		\binom{j}{\ell}
		\binom{k-j}{r-\ell}
		\binom{n-k-j}{r-\ell}.
	\end{equation}
	The multiplicity of $E_r(j)$ is
	\begin{equation}\label{eq:multiplicity}
		m_j=\binom nj-\binom n{j-1}.
	\end{equation}

	\subsection{Quantum coloring}
	
	A matrix $\mathbf{P}\in\Cc^{m\times m}$ is called a Hermitian projector if $\mathbf{P}^*=\mathbf{P}$ and $\mathbf{P}^2=\mathbf{P}$, where $\mathbf{P}^*$ denotes the conjugate transpose of $\mathbf{P}$. Cameron et al. \cite{Cameron2007} characterized the quantum chromatic number in terms of the existence of a suitable system of projectors.
	
	\begin{defn}\label{quantum-color}
		Let $G$ be a finite simple graph and let $c$ be a positive integer.  A quantum $c$-coloring of $G$ consists of a positive integer $m$ and Hermitian projectors
		\[
		\{\mathbf{P}_{v,\alpha}:v\in V(G),\ \alpha\in[c]\}\subseteq\Cc^{m\times m},
		\]
		satisfying the following two conditions.
		\begin{enumerate}[(1)]
			\item For every vertex $v\in V(G)$,
			\[
			\mathbf{P}_{v,\alpha}\mathbf{P}_{v,\beta}= \mathbf{O}\quad(\alpha\ne\beta),
			\quad
			\sum_{\alpha=1}^c \mathbf{P}_{v,\alpha}=\mathbf{I}_m.
			\]
			\item For every edge $uv\in E(G)$ and every color $\alpha\in[c]$,
			\[
			\mathbf{P}_{u,\alpha}\mathbf{P}_{v,\alpha}=\mathbf{O}.
			\]
		\end{enumerate}
		The quantum chromatic number $\chi_Q(G)$ is the least integer $c$ for which a quantum $c$-coloring exists.
	\end{defn}
	
The following notion is used to derive an upper bound on the quantum chromatic number.

\begin{defn}
	A modulus-one orthogonal representation of a graph $G$ in dimension $M$ is a map
	\[
	\rho:V(G)\longrightarrow \Cc^M
	\]
	satisfying the following conditions:
	\begin{enumerate}[(1)]
		\item each coordinate $\rho(v)_j$ of each vector $\rho(v)$ has modulus $1$.
		\item $\langle \rho(u),\rho(v)\rangle=0$ whenever $u$ and $v$ are adjacent in $G$.
	\end{enumerate}
	The minimum positive integer $M$ for which such a representation exists is denoted by $\xi(G)$.
\end{defn}

Using the parameter $\xi(G)$, Cameron et al. \cite{Cameron2007} established the following upper bound on the quantum chromatic number.
	
	\begin{lem}\label{lem:modulus-one-to-quantum}\cite{Cameron2007}
		For every graph $G$, we have $\chi_Q(G)\le \xi(G)$.
	\end{lem}
	
On the other hand, Elphick and Wocjan in \cite{Elphick2019} established the following Hoffman-type lower bound on the quantum chromatic number in terms of the eigenvalues of the graph.
	
	\begin{lem}\label{lem:sp-quantum}\cite{Elphick2019}
		Let $G$ be a graph with at least one edge.  Let
		\[
		\lambda_1\ge \lambda_2\ge\cdots\ge\lambda_N
		\]
		be the eigenvalues of the adjacency matrix of $G$.  Then
		\[
		\chi_Q(G)\ge 1+\frac{\lambda_1}{|\lambda_N|}.
		\]
		In particular, if $G$ is $\Delta$-regular and has smallest eigenvalue $\tau<0$, then
		\[
		\chi_Q(G)\ge 1+\frac{\Delta}{|\tau|}.
		\]
	\end{lem}
	
	\section{Quantum chromatic number of generalized Johnson graphs}
	
	\subsection{Upper bound on the quantum chromatic number} 
	
	In this subsection, we prove the existence of a modulus-one orthogonal representation of the generalized Johnson graph $J(n,k,k-r)$. Based on Lemma \ref{lem:modulus-one-to-quantum}, we provide the upper bound on the quantum chromatic number of $J(n,k,k-r)$. We begin with a general result on the upper bound with requirement $F_{n,r,t}(\theta)=0$. 
	
	\begin{prop}\label{prop-3-1}
		Let $n,k,r$ be integers with $1\le r\le \min\{k,n-k\}$. Let $t$ be an integer with $0\le t\le n$ and let $\theta\in\Rr$ be a real number. Let $J(n,k,k-r)$ be the generalized Johnson graph. Let
		\begin{equation}\label{prop-3-1-eq}
			F_{n,r,t}(\theta)
			=
			\sum_{\substack{0\le u\le r\\0\le v\le r\\0\le t-u-v\le n-2r}}
			\binom{r}{u}
			\binom{r}{v}
			\binom{n-2r}{t-u-v}
			e^{i\theta(v-u)}.
		\end{equation}
		If $F_{n,r,t}(\theta)=0$, then 
		\[
		\chi_Q(J(n,k,k-r))\le \binom nt.
		\]
	\end{prop}
	\begin{proof}
		We prove the upper bound of $\chi_Q(J(n,k,k-r))$ by designing a modulus-one orthogonal representation of $J(n,k,k-r)$.	
		
		For each $t$-subset $T\in\binom{[n]}t$ and each vertex $A\in\binom{[n]}k$, we define
		\[
		f_T(A)=e^{i\theta |A\cap T|},
		\]
		where $i=\sqrt{-1}$. Let $M=\binom nt$ and label the $t$-subsets of $[n]$ as $T_0,T_1,\ldots,T_{M-1}$. Let $\rho(A)=(f_{T_0}(A),f_{T_1}(A),\ldots,f_{T_{M-1}}(A))\in\Cc^M$. It is clear that each coordinate of $\rho(A)$ has modulus $1$.
		
		Given two adjacent vertices $A,B$ of $J(n,k,k-r)$, we have
		\begin{equation}\label{prop-3-1-eq-1}
			\langle \rho(A),\rho(B)\rangle=\sum_{T\in\binom{[n]}t}\overline{f_T(A)}f_T(B)=\sum_{T\in\binom{[n]}t}
			e^{i\theta(|B\cap T|-|A\cap T|)}.
		\end{equation}
		Since $|A|=|B|=k$ and $|A\cap B|=k-r$, the set $[n]$ is decomposed into the following four disjoint subsets:  
		\[
		A\cap B,
		\quad
		A\setminus B,
		\quad
		B\setminus A,
		\quad
		[n]\setminus(A\cup B),
		\]
		with $|A\setminus B|=|B\setminus A|=r$ and $|[n]\setminus(A\cup B)|=n-k-r$. Given a subset $X\subseteq [n]$, let
		\[
		\mathbf 1_X(x)
		=
		\begin{cases}
			1, & x\in X,\\
			0, & x\notin X.
		\end{cases}
		\]
		Then we have 
		\begin{equation}\label{prop-3-1-eq-2}
			|B\cap T|-|A\cap T|=\sum_{x\in T}\left(\mathbf 1_B(x)-\mathbf 1_A(x)\right).
		\end{equation}
		For each element $x\in [n]$, we define a weight function $w_x$ by
		\[
		w_x=
		\begin{cases}
			1, & x\in A\cap B,\\
			1, & x\notin A\cup B,\\
			e^{-i\theta}, & x\in A\setminus B,\\
			e^{i\theta}, & x\in B\setminus A.
		\end{cases}
		\]
		Hence, for each subset $T\subseteq[n]$, it follows from \eqref{prop-3-1-eq-2} that
		\begin{equation}\label{prop-3-1-eq-3}
			e^{i\theta\left(|B\cap T|-|A\cap T|\right)}
			=
			\prod_{x\in T} w_x.
		\end{equation}
		Combining \eqref{prop-3-1-eq-1} and \eqref{prop-3-1-eq-3}, we deduce that 
		\begin{equation}\label{prop-3-1-eq-4}
			\langle \rho(A),\rho(B)\rangle
			=
			\sum_{T\in\binom{[n]}{t}}
			e^{i\theta\left(|B\cap T|-|A\cap T|\right)}
			=
			\sum_{T\in\binom{[n]}{t}}
			\prod_{x\in T}w_x.
		\end{equation}
		
		We now build a generating function $\prod_{x\in[n]}(1+w_xz)$ for the sum over all $t$-subsets $T$. For each element \(x\in[n]\), the factor $1+w_xz$ has two choices. Choosing $1$ means that $x\notin T$, while choosing
		$w_xz$ means that $x\in T$. Expanding the generating function, we have 
		\[
		\prod_{x\in[n]}(1+w_xz)
		=
		\sum_{T\subseteq[n]}
		\left(\prod_{x\in T}w_x\right)z^{|T|}.
		\]
		According to \eqref{prop-3-1-eq-3} and \eqref{prop-3-1-eq-4}, $\langle \rho(A),\rho(B)\rangle$ is equal to the coefficient of $z^t$. Let $F_{n,r,t}(\theta)$ denote the coefficient of $z^t$. Since $|A\setminus B|=|B\setminus A|=r$, the values $w_x=1$, $w_x=e^{-i\theta}$ and  $w_x=e^{i\theta}$ occur respectively $n-2r$, $r$, $r$ times. Therefore, we have 
		\[
		\prod_{x\in[n]}(1+w_xz)=(1+z)^{n-2r}(1+e^{-i\theta}z)^r(1+e^{i\theta}z)^r,
		\]
		which implies that 
		\[
		F_{n,r,t}(\theta)
		=
		\sum_{\substack{0\le u\le r\\0\le v\le r\\0\le t-u-v\le n-2r}}
		\binom{r}{u}
		\binom{r}{v}
		\binom{n-2r}{t-u-v}
		e^{i\theta(v-u)}.
		\]
		By our assumption $F_{n,r,t}(\theta)=0$, we derive that $\langle \rho(A),\rho(B)\rangle=0$. Therefore $\{\rho(A): A \in \binom{[n]}k\}$ is a modulus-one orthogonal representation of $J(n,k,k-r)$ with dimension $M=\binom nt$. Using Lemma \ref{lem:modulus-one-to-quantum}, we have $\chi_Q(J(n,k,k-r))\le \binom nt$.
	\end{proof}
	
	Finding $t$ and $\theta$ satisfying $F_{n,r,t}(\theta)$ defined by \eqref{prop-3-1-eq}, we get the following two theorems. 
	
	\begin{thm}\label{thm-3-1}
		Let $n,k,r$ be integers with $1\le r\le \min\{k,n-k\}$ and $n \le 4r$. Let $J(n,k,k-r)$ be the generalized Johnson graph. Then we have $\chi_Q(J(n,k,k-r))\le n$.
	\end{thm}
	\begin{proof}
		According to \eqref{prop-3-1-eq}, letting $t=1$ and $\cos\theta = \frac{2r-n}{2r}$, we have $F_{n,r,1}(\theta)=n-2r+2r\cos\theta=0$. The desired result follows from Proposition \ref{prop-3-1}.
	\end{proof}

	\begin{thm}\label{thm-3-2}
		Let $n,k,r$ be integers with $1\le r\le \min\{k,n-k\}$ and $ 4r< n \le 4r+\left\lfloor\frac{1+\sqrt{16r+1}}{2}\right\rfloor$. Let $J(n,k,k-r)$ be the generalized Johnson graph. Then we have 
		\[
		\chi_Q(J(n,k,k-r))\le \binom{n}{2}.
		\]
	\end{thm}
	\begin{proof}
		letting $t=2$ in \eqref{prop-3-1-eq}, we have 
		\[
		F_{n,r,2}(\theta)=\binom{n-2r}{2}+2r(n-2r)\cos\theta+r+2r(r-1)\cos^2\theta.
		\]
		It is clear that the function $F_{n,r,2}(\theta)$ is continuous in \(\theta\). Putting $\theta=0$, we obtain 
		\[
		F_{n,r,2}(0)=\binom{n}{2}>0.
		\]
		Letting $\theta=\pi$, we get 
		\[
		F_{n,r,2}(\pi)=\frac{(n-4r)^2-n}{2}.
		\]
		Since $ 4r< n \le 4r+\left\lfloor\frac{1+\sqrt{16r+1}}{2}\right\rfloor$, we derive that $F_{n,r,2}(\pi)\le 0$. By the intermediate value theorem, there exists some
		$\theta\in [0,\pi]$ such that $F_{n,r,2}(\theta)=0$. Using Proposition \ref{prop-3-1}, we have
		\[
		\chi_Q(J(n,k,k-r))\le \binom{n}{2}.
		\]
	\end{proof}

	\subsection{Spectrum of $J(n,k,k-r)$ and lower bound on the quantum chromatic number}
	In this subsection, we discuss the spectrum of $J(n,k,k-r)$. By complement symmetry of $J(n,k,k-r)$, we assume that $k\le n/2$. Based on the spectrum of $J(n,k,k-r)$, we provide the lower bound on its quantum chromatic number.
	
	It was proved in \cite{Brouwer2018} that $E_r(1)$ is the smallest eigenvalue of $J(n,k,k-r)$ if and only if $r(n-1)\ge k(n-k)$. Using Lemma \ref{lem:sp-quantum}, we have the following lower bound on the quantum chromatic number of $J(n,k,k-r)$.
	
	\begin{thm}\label{l-b-1}
		Let $n,k,r$ be integers with $0< r\le k\le n/2$ and $r(n-1)\ge k(n-k)$. Let $J(n,k,k-r)$ be the generalized Johnson graph. Then we have
		\[
		\chi_Q(J(n,k,k-r))\ge 1+ \left\lceil
		\frac{k(n-k)}{rn-k(n-k)}
		\right\rceil.
		\]
	\end{thm}
	\begin{proof}
		By our assumption, $E_r(1)$ is the smallest eigenvalue of $J(n,k,k-r)$. Since the graph $J(n,k,k-r)$ is regular with degree
		\[
		\Delta_r=\binom{k}{r}\binom{n-k}{r},
		\]
		it follows from Lemma \ref{lem:sp-quantum} that
		\[
		\chi_Q(J(n,k,k-r))\ge 1+ \left\lceil\frac{\Delta_r}{|E_r(1)|}\right\rceil=1+\left\lceil\frac{k(n-k)}{rn-k(n-k)}
		\right\rceil.
		\]
	\end{proof}
	
	Let $\Delta_r=E_r(0)=\binom{k}{r}\binom{n-k}{r}$ and let $\theta_j=\frac{E_r(j)}{\Delta_r}$ for $j=0,\cdots,k$. It is easy to check that
	\begin{equation}\label{eq:theta1}
		\theta_1
		=1-\frac{rn}{k(n-k)}
		=\frac{k(n-k)-rn}{k(n-k)}.
	\end{equation}
	\begin{equation}\label{eq:theta2}
		\theta_2
		=
		\frac{r(r-1)}{(n-k)(n-k-1)}
		-\frac{2r(k-r)}{k(n-k)}
		+\frac{(k-r)(k-r-1)}{k(k-1)}.
	\end{equation}
	Subtracting \eqref{eq:theta1} from \eqref{eq:theta2} gives rise to 	
	\[
	\theta_2-\theta_1
	=
	\frac{r(n-2)\bigl(r(n-1)-k(n-k)\bigr)}
	{k(n-k)(k-1)(n-k-1)}.
	\]
	Based on the above equation, we deduce that $E_r(2)<E_r(1)$ if and only if $r(n-1)<k(n-k)$. It is not enough to state that $E_r(2)$ is the smallest eigenvalue of $J(n,k,k-r)$. In fact, $E_r(2)$ is not the smallest eigenvalue of $J(n,k,k-r)$ under given conditions, as the following example reveals. 
	\begin{example}
		Let $n=11$, $k=5$ and $r=2$. Then $r(n-1)<k(n-k)$. The eigenvalues of $J(11,5,3)$ are
		\[
		\begin{aligned}
			E_2(0)&=150, & E_2(1)&=40, & E_2(2)&=-5,\\
			E_2(3)&=-12, & E_2(4)&=-2, & E_2(5)&=10.
		\end{aligned}
		\]
		The smallest eigenvalue is $E_2(3)=-12$, not $E_2(2)=-5$.  
	\end{example}
	
	In what follows, we determine the conditions under which $E_r(2)$ is the smallest eigenvalue of $J(n,k,k-r)$.
	
	\begin{prop}\label{thm:E2-s-1}
		Let $a$ and $r$ be integers with $r\ge 5$ and $0\le a<\sqrt r$. Then the smallest eigenvalue of $J(4r,2r-a,r-a)$ is $E_r(2)$.
	\end{prop}
	\begin{proof}
		Let $J(n=4r,k=2r-a,r-a)$ be the generalized Johnson graph. Its degree is
		\[
		\Delta_r=\binom{2r-a}{r}\binom{2r+a}{r}.
		\]
		For $j=0,\cdots,k$, we let $\theta_j=\frac{E_r(j)}{\Delta_r}$. By $0\le a<\sqrt r$, we derive that $r(n-1)<k(n-k)$ which implies that $\theta_2<\theta_1$. 
		
		We next compare $\theta_2$ with $\theta_3$. Substituting into the Eberlein formula gives
		\[
		\theta_3-\theta_2=\frac{4r(r-1)\left(4r^3-6r^2+2r+a^2-a^4\right)}{(2r-a)(2r+a)(2r-a-1)(2r+a-1)(2r-a-2)(2r+a-2)}.
		\]
		Since $a<\sqrt r$, we get that
		\[
		4r^3-6r^2+2r+a^2-a^4>4r^3-7r^2+2r>0
		\]
		for $r\ge2$. This implies that $\theta_3>\theta_2$.
		
		It remains to prove that $\theta_2<\theta_j$ for all $j\ge4$. Let $A$ be the adjacency matrix of $J(n=4r,k=2r-a,r-a)$ and let $v=\binom{4r}{2r-a}$ be the number of vertices. Since $J(n=4r,k=2r-a,r-a)$ is $\Delta_r$-regular, we have $\tr(A^2)=v\Delta_r$. On the other hand, the eigenvalues of $A^2$ are $E_r(j)^2$, with the same multiplicities $m_j$ for $j=0,\cdots,k$. Then we have $\tr(A^2)=\sum_{j=0}^k m_jE_r(j)^2$. This implies that $m_jE_r(j)^2\le v\Delta_r$. Dividing by $m_j\Delta_r^2$ and taking the square root, we have 
		\begin{equation}\label{thm-6-eq-1}
			|\theta_j|=\left|\frac{E_r(j)}{\Delta_r}\right|
			\le
			\sqrt{\frac{v}{\Delta_rm_j}},~\mbox{with}~m_j=\binom{4r}{j}-\binom{4r}{j-1}.
		\end{equation}
		When $4\le j\le k\le 2r$, the multiplicities satisfy $m_j\ge m_4$. Furthermore, we have 
		\begin{equation}\label{thm-6-eq-2}
			m_4
			=\binom{4r}{4}-\binom{4r}{3}
			=\frac{4r(4r-1)(4r-2)(4r-7)}{24}
			\ge 3r^4,
		\end{equation}
		for $r\ge5$.
		
		We now estimate $v/\Delta_r$. Let 
		\[
		R(a)=\frac{v}{\Delta_r}=\frac{\binom{4r}{2r-a}}{\binom{2r-a}{r}\binom{2r+a}{r}}.
		\]
		Given an integer $a\ge0$, we obtain
		\[
		\frac{R(a+1)}{R(a)}=\frac{(2r-a)^2(r+a+1)}{(r-a)(2r+a+1)^2}.
		\]
		Subtracting the numerator from the denominator yields
		\[
		(r-a)(2r+a+1)^2-(2r-a)^2(r+a+1)=(2a+1)(r-a^2-a).
		\]
		If $0\le b<a<\sqrt r$, then $r-b^2-b>0$, which implies that $R(b+1)\le R(b)$ for every step needed to reach $a$. Therefore, we have $R(a)\le R(0)$. By Wallis' inequality in \cite{ChenQi2005} for central binomial coefficients, we have 
		\[
		\frac{4^r}{\sqrt{4r}} < \binom{2r}{r} < \frac{4^r}{\sqrt{3r+1}},
		\]
		when $r\ge 5$.
		This deduces that 
		\begin{equation}\label{thm-6-eq-3}
			R(a)\le R(0)=\frac{\binom{4r}{2r}}{\binom{2r}{r}^2}\le 2\sqrt r.
		\end{equation}
		
		For each $j\ge4$, using $r\ge5$, \eqref{thm-6-eq-2} and \eqref{thm-6-eq-3}, we have
		\[
		|\theta_j|\le\sqrt{\frac{2\sqrt r}{3r^4}}<\frac1{4r}.
		\]
		It follows from \eqref{eq:theta2} that 
		\begin{equation}\label{thm-6-eq-4}
			\theta_2=-\frac{4r^3-2r^2-2a^2r+a^2-a^4}{(2r-a)(2r+a)(2r-a-1)(2r+a-1)}.
		\end{equation}
		The inequality $|\theta_2|\ge\frac1{4r}$ is equivalent to $8r^3-4r^2+a^2-(4r+1)a^4\ge 0$. Thanks to $0\le a<\sqrt r$ and $r\ge5$, we have 
		\[
		8r^3-4r^2+a^2-(4r+1)a^4
		>8r^3-4r^2-(4r+1)r^2
		=4r^3-5r^2>0.
		\]
		Then we arrive at $	|\theta_j|<|\theta_2|$ for each $j\ge 4$. According to \eqref{thm-6-eq-4}, we have $\theta_2<0$. This implies $\theta_2<\theta_j$ for each $j\ge 4$. It is clear that $\theta_0=1>\theta_2$. Together with the explicit comparisons for $j=1,3$, this proves that $E_r(2)$ is the smallest eigenvalue of $J(n=4r,k=2r-a,r-a)$.
	\end{proof}
	
	Based on Lemma \ref{lem:sp-quantum} and Proposition \ref{thm:E2-s-1}, we obtain another lower bound on the quantum chromatic number of $J(n,k,k-r)$, as given below.
	
	\begin{thm}\label{l-b-2}
		Let $r\ge5$ and let $a$ be an integer with $0\le a<\sqrt r$. Then we have 
		\[
		\chi_Q(J(4r,2r-a,r-a))\ge 1+ \left\lceil
		\frac{(2r-a)(2r+a)(2r-a-1)(2r+a-1)}{4r^3-2r^2-2a^2r+a^2-a^4}
		\right\rceil.
		\]
	\end{thm}
	\begin{proof}
		According to Proposition \ref{thm:E2-s-1}, the smallest eigenvalue of $J(n=4r,k=2r-a,r-a)$ is $E_r(2)$. By Lemma \ref{lem:sp-quantum}, we have 
		\[
		\begin{aligned}
			\chi_Q(J(4r,2r-a,r-a))&\ge 1+ \left\lceil\frac{1}{|\theta_2|}\right\rceil\\
			&=1+ \left\lceil
			\frac{(2r-a)(2r+a)(2r-a-1)(2r+a-1)}{4r^3-2r^2-2a^2r+a^2-a^4}
			\right\rceil.
		\end{aligned}
		\]
	\end{proof}
	
	Using a similar technique as shown in the proof of Proposition \ref{thm:E2-s-1}, we obtain another family of generalized Johnson graphs with the smallest eigenvalue $E_r(2)$.
	
	\begin{prop}\label{thm:E2-s-2}
		Let $r$ and $t$ be positive integers with $t^2-t<4r-2$. Then, for all sufficiently large $r$, the smallest eigenvalue of $J(4r+t,2r,r)$ is $E_r(2)$.
	\end{prop}
	
	\begin{proof}
		Let $J(n=4r+t,k=2r,r)$ be the generalized Johnson graph. Let $\Delta_r=\binom{2r}{r}\binom{2r+t}{r}$ be the degree of $J(n=4r+t,k=2r,r)$. Let $\{E_r(j):j=0,\cdots,k\}$ be the eigenvalues of $J(n=4r+t,k=2r,r)$ and let $\theta_j=E_r(j)/\Delta_r$ for each $j=0,\cdots,k$.

		From \eqref{eq:theta1} and \eqref{eq:theta2}, we have
		\[
		\theta_1=1-\frac{r(4r+t)}{(2r)(2r+t)}=\frac{t}{2(2r+t)}
		\]
		and
		\begin{equation}\label{eq:theta2-perturbed}
			\theta_2=-\frac{4r^2-r t^2+3rt-2r+t^2-t}{2(2r-1)(2r+t)(2r+t-1)}.
		\end{equation}
		Since $t^2-t<4r-2$, we get that $\theta_1>0$ and $\theta_2<0$. By a direct calculation, we have
		\[
		\theta_3-\theta_2=\frac{r^2(r-1)(4r+t-4)(4r-t^2+t-2)}{(2r)(2r+t)(2r-1)(2r+t-1)(2r-2)(2r+t-2)}.
		\]
		By $t^2-t<4r-2$, we have $\theta_3>\theta_2$.

		Let $v=\binom{4r+t}{2r}$ be the number of vertices of $J(4r+t,2r,r)$. Let $m_j=\binom{4r+t}{j}-\binom{4r+t}{j-1}$ be the multiplicities of the eigenvalue $E_r(j)$. Our assumption $t^2-t<4r-2$ implies $t=O(\sqrt r)$. By Stirling's approximation in \cite{Robbins1955}, uniformly for $0<t<a\sqrt{r}$ with $a>0$ being some constant, we have 
		\[
		\frac{v}{\Delta}=\frac{\binom{4r+t}{2r}}{\binom{2r}{r}\binom{2r+t}{r}}=O(\sqrt r).
		\]
		
		For $4\le j \le 2r$, we have $m_j\ge m_4 \ge c r^4$ for some constant $c>0$. Then we have $m_j=\Omega(r^4)$ whenever $4\le j \le 2r$. Using a similar argument as stated in \eqref{thm-6-eq-1}, we have
		\[
		|\theta_j|=\left|\frac{E_r(j)}{\Delta}\right|
		\le
		\sqrt{\frac{v}{\Delta_rm_j}}=O(r^{-7/4}), ~\mbox{for}~ j \ge 4.
		\]
		By \eqref{eq:theta2-perturbed} and $t=O(\sqrt r)$, we deduce that $|\theta_2|= \Omega(r^{-3/2})$.

		Since $r^{-3/2}\gg r^{-7/4}$, for sufficiently large $r$, we have $|\theta_j|<|\theta_2|$ for $4\le j \le 2r$. The desired result follows as $\theta_2<0$.
	\end{proof}
	
	The following lower bound on the quantum chromatic number of $J(4r+t,2r,r)$ is derived from Lemma \ref{lem:sp-quantum} and Proposition \ref{thm:E2-s-2}
	
	\begin{thm}
		Let $r$ and $t$ be positive integers such that $t^2-t<4r-2$. Then, for all sufficiently large $r$, we have
		\[
		\chi_Q\left(J(4r+t,2r,r)\right)
		\ge 1+
		\left\lceil\frac{2(2r-1)(2r+t)(2r+t-1)}{4r^2-rt^2+3rt-2r+t^2-t}\right\rceil.
		\]
	\end{thm}
	\begin{proof}
		The desired result follows from Proposition \ref{thm:E2-s-2} and Lemma \ref{lem:sp-quantum}.
	\end{proof}
~~
~~

\subsection{Proofs of Theorems \ref{thm-exact-1} and \ref{thm-exact-2} }
~~
\medskip

\textbf{Proof of Theorem \ref{thm-exact-1}:} By Theorem \ref{thm-3-1} and Theorem \ref{l-b-1}, we have $\chi_Q(J(n,k,k-r))=n$ when $k\le n/2$. Since $J(n,k,k-r)$ is isomorphic to $J(n,n-k,n-k-r)$, we obtain the desired result.

\textbf{Proof of Theorem \ref{thm-exact-2}:} By $n=4r$ and Theorem \ref{thm-3-1}, we have $\chi_Q(J(4r,2r-a,r-a))\le 4r$. If $a^2=r$, then $r(4r-1)=(2r-a)(2r+a)$. According to Theorem \ref{thm-exact-1}, we obtain $\chi_Q(J(4r,2r-a,r-a))=4r$.
		
We next assume $a^2<r$. It follows from Theorem \ref{l-b-2} that
		\begin{equation}\label{thm12-eq-1}
			\chi_Q(J(4r,2r-a,r-a))\ge 1+ \left\lceil
			\frac{(2r-a)(2r+a)(2r-a-1)(2r+a-1)}{4r^3-2r^2-2a^2r+a^2-a^4}
			\right\rceil.
		\end{equation}
		Let
		\[
		\eta
		=
		\frac{4r^3-2r^2-2a^2r+a^2-a^4}
		{(2r-a)(2r+a)(2r-a-1)(2r+a-1)}.
		\]
		A direct computation gives
		\begin{equation}\label{thm12-eq-2}
			\frac1{4r-2}-\eta
			=
			\frac{a^2(4r-1)(a^2-1)}
			{(4r-2)(2r-a)(2r+a)(2r-a-1)(2r+a-1)}.
		\end{equation}
		Since $a\ge2$, the numerator in \eqref{thm12-eq-2} is positive. This implies that $\eta<\frac{1}{4r-2}$. By \eqref{thm12-eq-1}, we have
		\[
		\chi_Q(J(4r,2r-a,r-a))
		\ge
		1+\frac{1}{\eta}
		>4r-1.
		\]
		Combining the upper bound $\chi_Q(J(4r,2r-a,r-a))\le 4r$, we have $\chi_Q(J(4r,2r-a,r-a))= 4r$.
		
		Since $J(4r,2r-a,r-a)$ is isomorphic to $J(4r,2r+a,r+a)$, the same exact value holds for the complementary family.

\subsection{Proof of Theorem \ref{thm-s} }	
	\begin{lem}\label{c-c-g}
		Let $n,k,r$ be positive integers with $1\le r\le\min\{k,n-k\}$ and $|k-n/2|\leq \sqrt{n}/2$. Let $\epsilon$ be a real number with $0< \epsilon < 1/2 $. If $\epsilon n<2r< (1-\epsilon)n$, then we have
		\[
		\chi(J(n,k,k-r))
		\ge
		\frac{1}{e\sqrt{2}}\,\frac{2^{\delta n}}{\sqrt n}
		\]
		for some positive constants $\delta=\delta(\epsilon)$.
	\end{lem}
	\begin{proof}
		Let $J(n,k,k-r)$ be the generalized Johnson graph.  We identify every \(k\)-subset \(A\subseteq[n]\) with its binary incidence
		vector in \(\{0,1\}^n\). Given two \(k\)-subsets \(A,B\), their Hamming distance is $d_H(A,B)=|A\triangle B|$. Then two distinct vertices $A,B$ of $J(n,k,k-r)$ are adjacent if and only if $d_H(A,B)=2r$. Let \(\mathcal I\) be an independent set in \(J(n,k,k-r)\). Then no two distinct members of \(\mathcal I\), viewed as binary vectors, have Hamming distance \(2r\). Thus \(\mathcal I\) is a binary code with forbidden distance \(2r\). By \cite[Theorem 1.3]{Keevash2016}, if $\epsilon n<2r< (1-\epsilon)n$, then the independence number $\alpha(J(n,k,k-r))$ of $J(n,k,k-r)$ is smaller than or equal to $2^{(1-\delta)n}$ with some positive constant $\delta=\delta(\epsilon)$. Since $\chi(J(n,k,k-r))\ge\frac{|V(J(n,k,k-r))|}{\alpha(J(n,k,k-r))}$, we have
		\begin{equation}\label{lem-13-eq1}
			\chi(J(n,k,k-r))
			\ge
			\frac{\binom nk}{2^{(1-\delta)n}}.
		\end{equation}
		
		Let
		$
		H_2(x)=-x\log_2x-(1-x)\log_2(1-x)
		$
		be the binary entropy function. By
		\cite[Chapter~10, Lemma~7]{MacWilliamsSloane1977}, we obtain
		\begin{equation}\label{lem-13-eq3}
		\binom nk
		\ge
		\sqrt{\frac{n}{8k(n-k)}}\,
		2^{nH_2(k/n)}.
	\end{equation}
		Since $\left|k-\frac n2\right|\le\frac{\sqrt n}{2}$, it follows from $\ln x\le x-1$ for $x>0$ that 
		\[
		1-H_2\left(\frac kn\right) \leq \frac{4(\frac kn-\frac 1 2)^2}{\ln 2} \leq \frac{1}{n \ln 2}.
		\]
		Hence, we obtain 
		\[
		2^{nH_2(k/n)}\ge e^{-1}2^n.
		\]
		By $k(n-k)\le n^2/4$ and \eqref{lem-13-eq3}, we deduce that 
		\[
		\binom nk
		\ge
		\frac{1}{e\sqrt2}\,
		\frac{2^n}{\sqrt n}.
		\]
		The desired result follows from \eqref{lem-13-eq1}.
	\end{proof}

\textbf{Proof of Theorem \ref{thm-s} :}	
		Let $J(n,k,k-r)$ be the generalized Johnson graph defined by Theorem \ref{thm-exact-1}. Since $J(n,k,k-r)$ is isomorphic to $J(n,n-k,n-k-r)$, we assume that $k\le n/2$. Let $k=\frac n2-x$ for some $x\ge 0$. By $r(n-1)=k(n-k)$ and \(r\ge n/4\), we have 
		\[
		r(n-1)=\frac{n^2}{4}-x^2 \ge \frac{n^2}{4}-\frac n4,
		\]
		which implies that $x^2\le \frac n4$. Since $r=\frac{k(n-k)}{n-1}=\frac{n^2/4-x^2}{n-1}$, we have $\lim_{n\to\infty}\frac{2r}{n}=\frac{1}{2}$. For all sufficiently large $n$, we have 
		$
		\frac n3<2r<\frac{2n}{3}.
		$
		By Lemma \ref{c-c-g}, we have 
		\[
		\chi(J(n,k,k-r))
		\ge
		\frac{1}{e\sqrt{2}} \frac{2^{\delta n}}{\sqrt n}
		\]
		for some positive constants $\delta=\delta(1/3)$. If $n$ is large enough, we deduce that
		\[
		\chi(J(n,k,k-r))
		\ge
		\frac{1}{e\sqrt{2}}\frac{2^{\delta n}}{\sqrt n}>n=\chi_Q(J(n,k,k-r)).
		\]
		
		When the generalized Johnson graph is defined by Theorem \ref{thm-exact-2}, the desired result follows from a similar argument as shown above. 
		
\section*{Acknowledgement}
Part of this work was done when Xiwang Cao was visiting Nanyang Technological University. He thanks the host institution for its hospitality. He also thanks Chris Godsil for helpful discussions.
		
\section*{Declaration}
ChatGPT Pro was used for improving the presentation of the paper.

\end{document}